\documentclass{amsart}
\usepackage{epsfig}
\usepackage{graphicx}
\usepackage{tikz}
\usepackage{subfig}

\title[Caterpillar dualities]{Caterpillar dualities and regular languages}
\author[P. L. Erd\H{o}s]{P\'{e}ter L. Erd\H{o}s}
\author[C. Tardif]{Claude Tardif}
\author[G. Tardos]{G\'abor Tardos}

\address{Alfr\'ed R\'enyi Institute of Mathematics\\
 Hungarian Academy of Sciences\\
Budapest, P.O. Box 127, H-1364 Hungary}
\email{erdos.peter@renyi.mta.hu}

\address{Royal Military College of Canada \\
PO Box 17000 Station ``Forces'' \\
Kingston, Ontario, Canada, K7K 7B4 }
\email{Claude.Tardif@rmc.ca}

\address{Alfr\'ed R\'enyi Institute of Mathematics\\
 Hungarian Academy of Sciences\\
Budapest, P.O. Box 127, H-1364 Hungary}
\email{tardos.gabor@renyi.mta.hu}
\thanks{The first author's work was supported in part by the Hungarian NSF,
under contract NK 78439 and K 68262. The second author's work was supported
by grants from NSERC and ARP. The third author's work was supported in part by the NSERC grant 329527 by the Hungarian OTKA grants T-046234, AT048826 and NK-62321}


\newtheorem{lemma}{Lemma}[section]
\newtheorem{theorem}[lemma]{Theorem}

\newfont{\Bbbb}{msbm10 at 12pt}
\newcommand{\BA}{\mathbf A}
\newcommand{\BB}{\mathbf B}
\newcommand{\BC}{\mathbf C}
\newcommand{\BD}{\mathbf D}
\newcommand{\BP}{\mathbf P}
\newcommand{\BT}{\mathbf T}

\newcommand{\ig}{\mbox{\rm Inc} }
\newcommand{\Block}{\mbox{\rm Block} }

\begin{document}

%%%%%%%  abstract %%%%%%%%%%%%

\begin{abstract}
\noindent We characterize obstruction sets in caterpillar dualities in terms of regular languages, and give a construction of the dual of a regular family of caterpillars. We show that these duals correspond to the constraint satisfaction problems definable by a monadic linear Datalog program with at most one EDB per rule.
\end{abstract}

\date{\today}
\maketitle

\section{Introduction}
A {\em homomorphism duality} is a couple $(\mathcal{O}, \BD)$ where $\BD$ is a relational structure and  $\mathcal{O}$ is a family relational structures of the same type, such that the following holds.
\begin{quote}
For any given relational structure $\BA$, there exists a homomorphism from $\BA$ to $\BD$ if and only if there is no homomorphism from any member $\BT$ of $\mathcal{O}$ to $\BA$.
\end{quote}
Significant dualities typically correspond to efficient algorithms for constraint satisfaction problems. These include finite dualities (where the family $\mathcal{O}$ is finite), tree dualities (where $\mathcal{O}$ is a family of trees) and bounded treewidth dualities (where $\mathcal{O}$ is a family of structures with bounded treewidth). More examples are discussed in~\cite{bkl}.

``Characterizing dualities'' may refer to two distinct types of problems.
\begin{itemize}
\item Characterizing targets: Deciding, given a structure $\BD$, whether there exists a family $\mathcal{O}_{\BD}$ of structures in a given class (e.g. trees) such that $(\mathcal{O}_{\BD}, \BD)$ is a duality.
\item Characterizing obstruction sets: Deciding, given a family $\mathcal{O}$, whether there exists a structure $\BD_{\mathcal{O}}$ such that $(\mathcal{O}, \BD_ {\mathcal{O}})$ is a duality.
\end{itemize}
The two problems are different. In the case of finite dualities, the characterization of obstruction sets was obtained in 2000 (\cite{nt}), and that of targets in 2007 (\cite{llt}).
The problem of characterizing targets was solved in 1998 (\cite{fv}) for tree dualities, and recently in 2009 (\cite{bk}) for bounded treewidth dualities. Characterizing obstruction sets remains an open problem both for tree duality and bounded treewidth duality.

The difficulty in characterizing obstruction sets may depend on how the obstructions are represented. In the case of finite dualities, an explicit description of the obstructions is always possible. For infinite families of obstructions, fragments of the Datalog language have proved to be an efficient tool to describe families of obstructions implicitly, through their  homomorphic images. The structures with tree duality and bounded treewidth duality all have obstruction sets that can be described in Datalog.

In \cite{cdk}, Carvalho, Dalmau and Krokhin introduced caterpillar dualities as the dualities $(\mathcal{O}, \BD)$ where $\mathcal{O}$ is describable in the smallest natural recursive fragment of Datalog, namely ``monadic linear Datalog with at most one EDB per rule''
(see Section~\ref{datalog}). They proved that the corresponding targets $\BD$ are precisely those which are homomorphically equivalent to a structure with lattice polymorphisms, and that they are recognizable by the existence of a homomorphism of a given superstructure
$\BC(\BD)$ to $\BD$ (see Section~\ref{constructions}).

The purpose of the present paper is to complement the work of Carvalho, Dalmau and Krokhin by solving the characterization of obstructions problem for caterpillar dualities. We will consider a representation of caterpillars by words over a suitable alphabet, and show that caterpillar dualities correspond to regular languages. In particular, this shows that every program in ``monadic linear Datalog with at most one EDB per rule'' describes the obstruction set of a caterpillar duality. This extends some methods developed in \cite{ett} to study antichain dualities for digraphs.  The case of general tree dualities is considered
in \cite{eptt}

We will provide the necessary background in the next section, and prove our main result in Section~\ref{caterpillars}. The link with Datalog is given in Section~\ref{datalog}, and relevant constructions and extensions are discussed in Section~\ref{constructions}

\section{Preliminaries}

\smallskip \noindent{\em Relational structures.}
A {\em type} is a finite set $\sigma = \{R_1,\dots,R_m\}$ of {\em relation symbols}, each with  an {\em arity} $r_i$ assigned to it. A $\sigma$-structure is a relational structure $\BA = \langle A;R_1(\BA),\dots,R_m(\BA)\rangle$ where $A$ is a non-empty set called the {\em universe} of $\BA$, and $R_i(\BA)$ is an $r_i$-ary relation on $A$ for each $i$. The elements of $R_i(\BA)$, $1\leq i \leq m$ will be called {\em hyperedges} of $\BA$. By analogy with the graph theoretic setting, the universe of $\BA$ will also be called its vertex-set, denoted $V(\BA)$.

A $\sigma$-structure $\BA$ may be described by its bipartite {\em incidence multigraph} $\ig(\BA)$ defined as follows. The two parts of $\ig(\BA)$ are $V(\BA)$ and $\Block(\BA)$, where
$$
\Block(\BA) = \{ (R,(x_1, \ldots, x_{r})) : R \in \sigma \mbox{ has arity $r$ and }
(x_1, \ldots, x_{r}) \in R(\BA) \},
$$
and with edges $e_{a,i,B}$ joining $a \in V(\BA)$ to $B=(R,(x_1, \ldots, x_{r})) \in \Block(\BA)$ when $x_i = a$. Thus, the degree of $B=(R,(x_1, \ldots, x_{r}))$ in $\ig(\BA)$
is precisely $r$. Here ``degree'' means number of incident edges rather than number of neighbors because parallel edges are possible: If $x_i = x_j = a \in V(\BA)$, then $e_{a,i,B}$ and $e_{a,j,B}$ both join $a$ and $B$. An element $a \in V(\BA)$  is called a {\em leaf} if it
has degree one in $\ig(\BA)$, and a {\em non-leaf} otherwise.  Similarly, a block of $\BA$ is called {\em pendant} if it is incident to at most one non-leaf, and {\em non-pendant} otherwise. A $\sigma$-structure $\BT$ is called a {\em $\sigma$-tree} (or {\em tree} for short) if $\ig(\BT)$ is a (graph-theoretic) tree, that is, it is connected and has no cycles or parallel edges. A $\sigma$-tree is called a {\em path} if it has at most two pendant blocks. A $\sigma$-tree is called a {\em caterpillar} if it is either a path or it can be turned into a path by removing all its pendant blocks (and the leaves attached to them).

\smallskip \noindent{\em Homomorphisms.}
For $\sigma$-structures $\BA$ and $\BB$, a {\em homomorphism} from $\BA$ to $\BB$ is a map $f: V(\BA) \mapsto V(\BB)$ such that $f(R_i(\BA)) \subseteq R_i(\BB)$ for all $i=1,\dots,m$, where for any relation $R \in \sigma$ of arity $r$ we have
$$
f(R) = \{(f(x_1),\dots,f(x_r)):(x_1,\dots,x_r) \in R\}.
$$
We write $\BA \rightarrow \BB$ if there exists a homomorphism from $\BA$ to $\BB$, and $\BA \not \rightarrow \BB$ otherwise. We write $\BA \leftrightarrow \BB$ when $\BA \rightarrow \BB$ and $\BB \rightarrow \BA$; $\BA$ and $\BB$ are then called {\em homomorphically equivalent}. For a finite structure $\BA$, we can always find a structure $\BB$ such that $\BA \leftrightarrow \BB$ and the cardinality of $V(\BB)$ is minimal with respect to this property. It is well known (see \cite{nt}) that any two such structures are isomorphic. We then call $\BB$ the {\em core} of $\BA$.

\smallskip \noindent{\em Automata.}

When the type $\sigma$ consists only of binary relations, a $\sigma$-structure $\BA$ is an edge-labeled directed graph. If we specify sets $I, T \subseteq V(\BA)$ of initial and terminal states respectively, we get a nondeterministic automaton $(\BA,I,T)$. The type $\sigma$
is then viewed as an alphabet. A word $w \in \sigma^*$ naturally corresponds to a directed $\sigma$-path $P_w$ with $|w|$ edges with labels successively specified by the letters of $w$. A walk is a homomorphism $\phi: P_w \rightarrow \BA$. If $\phi$ maps the first and last vertices of $P_w$ to vertices in $I$ and $T$ respectively, then the word $w$ is {\em accepted} by $(\BA, I, T)$. The set of such words is called the {\em language accepted by $(\BA, I, T)$}.

We recall a few basic facts from automata theory. The reader is referred to standard references (e.g. \cite{sipser}) for a thorough treatment. A language $\mathcal{L} \subseteq \sigma^*$ is called {\em regular} if it is the language accepted  by some nondeterministic automaton. It is well known that a language is regular if and only if it can be described by a ``regular expression'', that is, an expression constructed from letters in $\sigma$ using unions, concatenation and the star operation. Regular languages are also preserved by other basic operations such as intersection and complementation.

An automaton $(\BA, I, T)$ is called {\em deterministic} if $I$ is a singleton and for every $a \in V(\BA)$ and $R \in \sigma$, there is a unique $b \in V(\BA)$ such that $(a,b) \in R(\BA)$. It  is well known that for every non-deterministic automaton $(\BA,I,T)$, there exists a deterministic automaton $\Delta(\BA,I,T)$ which accepts the same language.

\section{Caterpillars} \label{caterpillars}

Graph-theoretic caterpillars consist of a path ``body'' to which are connected a number of pendant ``leg'' edges. Similarly, the non-leaves of a general caterpillar (with at least two blocks) can be linearly ordered $x_1, \ldots, x_n$ such that $x_{i},x_{i+1}$ are incident to one common non-pendant block $B_i$ for $i = 1, \ldots, n-1$. The remaining blocks of $\BT$ are pendant, and each of them is incident to one of $x_1, \ldots, x_n$. In this section we present a way to represent caterpillars by words over a suitable alphabet.

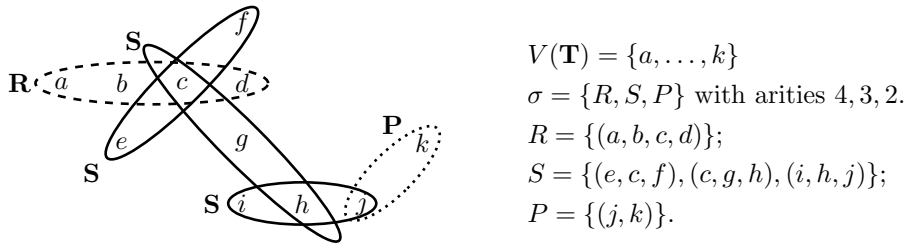
\begin{figure}[h]\label{f:structure}
\hspace{-3cm}
\begin{minipage}{5.5cm}
\begin{tikzpicture}[scale=.8]
\node at (1.3,2) {$\mathbf{R}$};
\node at (2,2) {$a$};
\node at (3,2) {$b$};
\node at (4,2) {$c$};
\node at (5,2) {$d$};
\draw [line width=1pt,dashed] (3.5,2) ellipse (55pt and 10pt);
\node at (3,1) {$e$};
\node at (5,3) {$f$};
\draw [rotate around={45:(4,2)},line width=1pt] (4,2) ellipse (50pt and 10pt);
\node at (2.5,.5) {$\mathbf{S}$};
\node at (5,1) {$g$};
\node at (6,0) {$h$};
\draw [rotate around={135:(5,1)},line width=1pt] (5,1) ellipse (65pt and 10pt);
\node at (3.2,2.7) {$\mathbf{S}$};
\node at (5,0) {$i$};
\node at (7,0) {$j$};
\draw [line width=1pt] (6,0) ellipse (35pt and 10pt);
\node at (4.5,0) {$\mathbf{S}$};
\node at (8,1) {$k$};
\draw [rotate around={-45:(7.5,.5)},line width=1pt,dotted] (7.5,.5) ellipse (10pt and 30pt);
\node at (7.5,1.3) {$\mathbf{P}$};
\end{tikzpicture}
\end{minipage}
\qquad
\begin{minipage}{3cm}
\begin{eqnarray*}
&&V(\mathbf{T})=\{a,\ldots,k\}\\
&&\sigma=\{R,S,P\} \mbox{ with arities }  4,3,2. \\
&&R=\{(a,b,c,d)\}; \\
&&S=\{(e,c,f), (c,g,h), (i,h,j)\};\\
&&P=\{(j,k)\}.
\end{eqnarray*}
\end{minipage}
\caption{The caterpillar  $\mathbf{T}$}
\end{figure}

\begin{figure}\label{f:Inc}
\begin{tikzpicture}[scale=.85]
\node  [shape=rectangle,draw] (p1) at ( 2.7,2) {$R(a,b,c,d)$};
\node  [shape=rectangle,draw] (p2) at ( 5,2) {$S(e,c,f)$};
\node  [shape=rectangle,draw] (p3) at ( 7,2) {$S(c,g,h)$};
\node  [shape=rectangle,draw] (p4) at ( 9,2) {$S(i,h,j)$};
\node  [shape=rectangle,draw] (p5) at ( 11,2) {$P(j,k)$};
\node  [shape=circle,draw] (p6) at ( 2,0) {$a$};
\node  [shape=circle,draw] (p7) at ( 3,0) {$b$};
\node  [shape=circle,draw] (p8) at ( 4,4) {$c$};
\node  [shape=circle,draw] (p9) at ( 5,0) {$d$};
\node  [shape=circle,draw] (p10) at ( 6,0) {$e$};
\node  [shape=circle,draw] (p11) at ( 7,0) {$f$};
\node  [shape=circle,draw] (p12) at ( 8,0) {$g$};
\node  [shape=circle,draw] (p13) at ( 9,4) {$h$};
\node  [shape=circle,draw] (p14) at ( 10,0) {$i$};
\node  [shape=circle,draw] (p15) at ( 11,4) {$j$};
\node  [shape=circle,draw] (p16) at ( 12,0) {$k$};
\draw [line width=1pt]  (p1) -- (p6);
\draw [line width=1pt]  (p1) -- (p7);
\draw [line width=1pt]  (p1) -- (p8);
\node  [fill=gray!40] at (2,3) {$e_{c,3,R(a,b,c,d)}$};
\draw [line width=1pt]  (p1) -- (p9);
\draw [line width=1pt]  (p2) -- (p8);
\draw [line width=1pt]  (p2) -- (p10);
\draw [line width=1pt]  (p2) -- (p11);
\draw [line width=1pt]  (p3) -- (p8);
\draw [line width=1pt]  (p3) -- (p12);
\draw [line width=1pt]  (p3) -- (p13);
\draw [line width=1pt]  (p4) -- (p14);
\draw [line width=1pt]  (p4) -- (p13);
\draw [line width=1pt]  (p4) -- (p15);
\draw [line width=1pt]  (p5) -- (p15);
\node  [fill=gray!40] at (12,3) {$e_{j,1,P(j,k)}$};
\draw [line width=1pt]  (p5) -- (p16);
\node [fill=gray!40] (p17) at (9,6) {non-leaf elements};
\draw [dashed] (p17) -- (p8);
\draw [dashed] (p17) -- (p13);
\draw [dashed] (p17) -- (p15);
\node [fill=gray!40] (p18) at (5,6) {non-pendant blocks};
\draw [dotted] (p18) -- (p3);
\draw [dotted] (p18) -- (p4);
\end{tikzpicture}
\caption{The bipartite graph $\mathrm{Inc}(\mathbf{T})$}
\end{figure}
Given a type $\sigma$, we define $\sigma_2$ as follows: For every $R \in \sigma$ of arity $k$ and for every $(i,j) \in \{ 1, \ldots, k\}^2$, $\sigma_2$ contains the symbol $R^{(i,j)}$.
Thus $\sigma_2$ can be viewed as an alphabet or as a type consisting of binary relations.

As an alphabet, $\sigma_2$ allows to represent $\sigma$-caterpillars in a natural way: If
$\BT$ is a $\sigma$-caterpillar and $x_1, \ldots, x_n$ are its non-leaves with their natural ordering, then $\BT$ corresponds to the $\sigma_2$-word
$$
X_1 L_1 X_2 L_2 X_3 \cdots X_{n-1} L_{n-1} X_n,
$$
where $X_i$ is the concatenation of all $R^{(j,j)}$s such that $\BT$ has a pendant block
$(R,(a_1, \ldots, a_k))$ with $a_j = x_i$, and $L_i$ is $R^{(j,k)}$ such that $\BT$ has a non-pendant block $(R,(a_1, \ldots, a_{\ell}))$ with $a_j = x_i$, $a_k = x_{i+1}$. A caterpillar consisting of a single block $(R,(a_1, \ldots, a_k))$ can be represented by any letter of the letters $R^{(i,j)}$, and the caterpillar consisting of one vertex and no blocks is represented by the empty word. In general, different words may represent the same caterpillar. However a caterpillar may be retrieved from any word representing it. This retrieval is essentially a functor from the category of $\sigma_2$-structures (where $\sigma_2$ is interpreted as a type) to that of $\sigma$ structures, as detailed below.

\begin{figure}[h]\label{f:new}
\centering
\subfloat [] {
\fbox{
\begin{minipage}{5.5cm}
$\sigma^2=\left \{ R^{11},\ldots,R^{14}, R^{21},\ldots, R^{44}, \right .$\\
$\phantom{\sigma^2= \{ } S^{11},S^{12},S^{13},S^{21},\ldots,S^{33},    $\\
$\phantom{\sigma^2=  \{ } \left . P^{11},P^{12},P^{21},P^{22} \right \}$
    \end{minipage}
    }}
    \qquad
\subfloat [] {
\fbox{
\begin{minipage}{3cm}
$R^{33}S^{22}S^{13}S^{23}P^{11}$   \\
\ \\
$S^{12}R^{33}S^{13}S^{21}P^{12}$
\end{minipage}
}}
\caption{(a) The alphabet $\sigma^2(\mathbf{T})$
     (b) two of the 28 words  describing  $\mathrm{Inc}(\mathbf{T})$
}

\end{figure}

There is a natural functor $\beta$ which takes a $\sigma$-structure $\BA$ and produces a corresponding $\sigma_2$-structure $\beta(\BA)$: We put $V(\beta(\BA)) = V(\BA)$,
and for $R \in \sigma$ and $(x_1, \ldots, x_k) \in R(\BA)$, we put $(x_i,x_j) \in R^{(i,j)} (\beta(\BA))$ for all $(i,j) \in \{ 1, \ldots, k\}^2$. The functor $\beta$ is a right adjoint in the sense of~\cite{ft,pultr}, thus there exists a corresponding left adjoint $\beta^*$ such that for a $\sigma_2$-structure $\BA$ and a $\sigma$-structure $\BB$, we have
$$
{\BA \rightarrow \beta(\BB)}  \Leftrightarrow {\beta^*(\BA) \rightarrow \BB}.
$$
The $\sigma$-structure $\beta^*(\BA)$ is constructed as follows. We first construct an auxiliary structure $\beta^*(\BA)^+$. For each element $x \in V(\BA)$,  $V(\beta^*(\BA)^+)$ contains a corresponding (isolated) element $x'$, and for each $(x,y) \in R^{(i,j)}(\BA)$,
$V(\beta^*(\BA)^+)$ contains additional elements $x_1, \ldots, x_k$ (where $R \in \sigma$ has arity $k$) and the hyperedge $(x_1, \ldots, x_k) \in R(\beta^*(\BA)^+)$. $\beta^*(\BA)$ is then the quotient $(\beta^*(\BA)^+)/\sim$ obtained through natural identifications. That is,
for $(x,y) \in R^{(i,j)}(\BA)$ and the corresponding $(x_1, \ldots, x_k) \in R(\beta^*(\BA)^+)$, $\sim$ identifies $x_i$ with $x'$ and $x_j$ with $y'$.

Note that the construction of $\beta^*(\BA) = (\beta^*(\BA)^+)/\sim$ may identify elements $x', y'$ that correspond to distinct elements of $x, y \in V(\BA)$. This happens when for $x, y \in V(\BA)$, $x \neq y$, there is some $R^{(i,i)} \in \sigma_2$ such that $(x,y) \in R^{(i,i)}$. In particular, a $\sigma$-caterpillar $\BT$ is described by a $\sigma_2$-word $w$, with letters of type $R^{(i,i)}$ describing its legs. In turn, $w$ naturally corresponds to the  $\sigma_2$-path $\BP_w$ with $|w|+1$ elements successively joined by the relations indicated by the letters of $w$. We then have $\beta^*(\BP) \simeq \BT$. The adjunction property between $\beta$ and $\beta^*$ implies the following.

\begin{lemma} \label{imisreg}
Let $\sigma$ be a type and $\BA$ a $\sigma$-structure. Then the family of $\sigma_2$-words describing the caterpillars that admit homomorphisms to $\BA$ is a regular language.
\end{lemma}
\begin{proof}
Let $\BT$ be a $\sigma$-caterpillar, $w$ a word describing it and $\BP_w$ the $\sigma_2$-path corresponding to $w$. Then the adjunction property yields
$$
{\BP_w \rightarrow \beta(\BA)}  \Leftrightarrow {\beta^*(\BP_w) \rightarrow \BA}.
$$
with $\beta^*(\BP_w) \simeq \BT$. Since $\beta(\BA)$ can be viewed as a  nondeterministic automaton with all states being initial and terminal, this shows that the corresponding words $w$ indeed constitute a regular language.
\end{proof}
Since the complement of a regular language is again regular, the family of caterpillar obstructions of any $\sigma$-structure $\BA$ is again described by a regular language.

\begin{theorem} \label{reghasdual}
Let $\sigma$ be a type, $\mathcal{L}$ a regular language over $\sigma_2$ and $\mathcal{O}$ the family of $\sigma$-caterpillars represented by $\mathcal{L}$.
Then there exists a $\sigma$-structure $\BA$ such that $(\mathcal{O},\BA)$ is a homomorphism duality.
\end{theorem}

\begin{proof}
Let $(\BD,I,T)$ be a deterministic automaton which recognizes  $\mathcal{L}$. We define the structure $\BA = \Gamma(\BD,I,T)$ as follows. $V(\BA)$ is the set of subsets of $V(\BD)$ containing the initial state but none of the terminal states. For a relation $R \in \sigma$ of arity $k$, $R(\BA)$ is defined as follows: We put $(X_1, \ldots, X_k) \in R(\BA)$ if for all $(i,j) \in  \{ 1, \ldots, k\}^2$ and for all $a \in X_i$, the unique $b$ such that $(a,b) \in R^{(i,j)}(\BD)$ is in $X_j$.

Let $\BB$ be a structure such that no caterpillar described by $\mathcal{L}$ admits a  homomorphism to $\BB$. Let $w$ be a word over $\sigma_2$ such that there exists a homomorphism $\phi: \beta^*(\BP_w) \rightarrow \BB$. Then, $\phi$ induces a homomorphism $\phi_2: \BP_w \rightarrow \beta(\BB)$, and we denote $b_{w,\phi}$ the image of the last vertex of $\BP_w$ under $\phi_2$. Also, there is a unique homomorphism of $\BP_w$ to $\BD$ mapping the first element to the start state, and we denote $d_w$ the image of the last vertex of $\BP_w$. Using every possible $w$ and $\phi: \BT \rightarrow \BB$ we define a map $\psi: V(\BB) \rightarrow \mathcal{P}(D)$ as follows. For an element $b$ of $\BB$, $\psi(b)$ is the set of all elements $d_w$ such that $b = b_{w,\phi}$. Then $\psi(b)$ always contains the start state (because the empty word represents the  one-element caterpillar with no hyperedges, which can be mapped to $b$) and never a terminal state (because $d_w$ can never be a terminal state). Thus $\psi$ is a map from $V(\BB)$ to $V(\BA)$. We prove that it is a homomorphism of $\BB$ to $\BA$. Let $R$ be a relation in $\sigma$ of arity $k$, and $(b_1, \ldots, b_k) \in R(\BB)$. For $(i,j) \in \{1, \ldots, k\}^2$ and $d \in \psi(b_i)$, there exists a word $w$ such that $d_w = d$ and there exists a homomorphism $\phi: \beta^*(\BP_w) \rightarrow \BB$ such that $b_{w,\phi} = b_i$. By appending $R^{(i,j)}$ to $w$, we get a new word $w'$ such that $\phi: \beta^*(\BP_w) \rightarrow \BB$ naturally extends to $\phi': \beta^*(\BP_{w'}) \rightarrow \BB$, with $b_{w,\phi'} = b_j$. Therefore the unique element $d_{w'}$ such that $(d_{w},d_{w'}) \in R^{(i,j)}$ is in $\psi(b_j)$. This shows that $\psi$ is a homomorphism.

Therefore, if no caterpillar described by $\mathcal{L}$ admits a homomorphism to $\BB$, then $\BB$ admits a homomorphism to $\BA$. It remains to prove that no caterpillar described by $\mathcal{L}$ admits a homomorphism to $\BA$. For $w \in \mathcal{L}$, suppose that there exists a homomorphism $\phi:  \beta^*(\BP_{w'}) \rightarrow \BA$. This corresponds to a homomorphism $\phi_2: \BP_w \rightarrow \beta(\BA)$. Since the start state is in the image of the first element of $\BP_w$, a terminal is in the image of its last element, which is impossible.
\end{proof}
According to Theorem \ref{reghasdual}, for every regular $\sigma_2$-language $\mathcal{L}$, there exists a duality $(\mathcal{O},\BA)$ such that $\mathcal{O}$ is the family of caterpillars described by $\mathcal{L}$. However $\mathcal{L}$ may be smaller than the set $\mathcal{L}^+$ of all words describing caterpillar obstructions to $\BA_{\mathcal{L}}$; however by Lemma~\ref{imisreg}, $\mathcal{L}^+$ is also regular (since its complement
is regular). Between $\mathcal{L}$ and $\mathcal{L}^+$ there are usually non-regular languages which also describe complete set of obstructions to $\BA$. There may even be such non-regular languages that do not contain $\mathcal{L}$. Therefore, the complete characterization of obstruction sets for caterpillar dualities may be stated as follows:
\begin{theorem}
Let $\mathcal{L}$ be a $\sigma_2$-language, $\mathcal{O}$ the family of $\sigma$-caterpillars described by $\mathcal{L}$, $\mathcal{O}^+$ the family of $\sigma$-caterpillars which contain homomorphic images of members of $\mathcal{O}$ and $\mathcal{L}^+$ the collection of words describing these caterpillars. Then there exists a duality $(\mathcal{O},\BA)$ if and only if $\mathcal{L}^+$ is regular.
\end{theorem}

\section{Caterpillar Datalog programs} \label{datalog}

A {\em caterpillar Datalog program} is a ``monadic linear Datalog program with at most one EDB per rule'', that is, a set of rules of the form
\begin{equation} \label{datrule}
a \in \rho_i \leftarrow \mbox{$b \in \rho_j$ and $(x_1, \ldots, x_k) \in R$ with $x_m = a, x_n = b$}.
\end{equation}
Here $R$ is a relation in a type $\sigma$ of arity $k$ (called an extensional database or EDB), and $\rho_i, \rho_j$ are unary auxiliary relations that are not in $\sigma$ and that will be defined recursively (they are called intensional databases or IDBs). The auxiliary relations are monadic, that is, unary, and the program is ``linear'' since at most one auxiliary relation is used in the condition on the right side of the arrow. (See \cite{fv} for a description of general Datalog programs.) In addition, the first rule is a formal initialization:
\begin{equation} \label{datinit}
a \in \rho_1 \leftarrow a = a,
\end{equation}
and there are terminal rules of the form
\begin{equation} \label{datterm}
\mbox{goal} \leftarrow \mbox{$a \in \rho_i$}.
\end{equation}

A Datalog program is usually seen as a way to construct unary relations $\rho_1, \rho_2, \ldots$ in a $\sigma$-structure $\BB$ recursively, by a repeated application of the rules that apply, until a certain ``goal'' is achieved. Note that all the rules can be rewritten in terms of the type $\sigma_2$: The rule \ref{datrule} can be written
\begin{equation} \label{datrulesigma2}
a \in \rho_i \leftarrow \mbox{$b \in \rho_j$ and $(b,a) \in R^{(n,m)}$}.
\end{equation}
In this modified form, the program can be executed in $\beta(\BB)$. We see that the ``goal'' is achieved when a certain $\sigma_2$-walk is found in $\beta(\BB)$, which corresponds to finding a homomorphic image of the corresponding caterpillar in $\BB$.

Therefore, a caterpillar Datalog program will achieve its goal on the structures which contain homomorphic images of caterpillars belonging to a certain family. To see that this family is regular, we consider the nondeterministic automaton $(\BC,I,T)$ of type $\sigma_2$ described by the rules of the programs: $V(\BC)$ is the set of IDB's of the program, and for each rule
$$
a \in \rho_i \leftarrow \mbox{$b \in \rho_j$ and $(b,a) \in R^{(n,m)}$}
$$
we put $(\rho_j,\rho_i) \in R^{(n,m)}(\BC)$. We put $I = \{\rho_1\}$, and the terminal states are the states $\rho_i$ appearing in terminal rules. Thus a goal-achieving derivation in a structure $\BB$ must correspond to a word accepted by $(\BC,I,T)$, and the family of such words is regular. Combining this with Theorem \ref{reghasdual} we get the following.
\begin{theorem}
For every caterpillar Datalog program, there exists a structure $\BA$ such that an input structure $\BB$ admits a homomorphism to $\BA$ if and only if the program does not achieve its goal on $\BB$.
\end{theorem}

\section{Construction and characterization of duals} \label{constructions}
For a type $\sigma$, a regular $\sigma_2$-language $\mathcal{L}$ may be described by a regular expression, an automaton (deterministic or nondeterministic) which recognizes it or a caterpillar Datalog program. The previous section explains how to convert a caterpillar Datalog program into a nondeterministic automaton which recognizes the same language. We refer to \cite{sipser} for the conversion from regular expression to automaton, and for the construction $\Delta$ which takes a nondeterministic automaton $(\BB,I,T)$ and constructs a deterministic automaton $(\BD,I',T')  = \Delta(\BB,I,T)$ which accepts the same language.
Thus, if the regular $\sigma_2$-language $\mathcal{L}$ is recognized by the automaton $(\BB,I,T)$, then the corresponding caterpillar duality is $(\mathcal{O},\BA)$, where $\mathcal{O}$ is the family of $\sigma$-caterpillars described by $\mathcal{L}$ and
$\BA = \Gamma \circ \Delta(\BB,I,T)$, $\Gamma$ being the construction described in the proof of Theorem~\ref{reghasdual}.

Now for any $\sigma$-structure $\BA$, $(\beta(\BA),V(\BA),V(\BA))$ is a nondeterministic automaton which recognizes the $\sigma_2$-language of words describing caterpillars which admit a homomorphism to $\BA$, and $\Delta(\beta(\BA),V(\BA),V(\BA))$ is a deterministic automaton which serves the same purpose. Let  $\Delta^*(\beta(\BA), V(\BA),V(\BA))$ be the deterministic automaton obtained from $\Delta(\beta(\BA), V(\BA),V(\BA))$ by interchanging the set of terminal states with its complement. Then $\Delta^*(\beta(\BA),V(\BA),V(\BA))$ is a deterministic automaton which recognizes
the $\sigma_2$-language of words describing the set $\mathcal{O}$ of caterpillars
which do not admit a homomorphism to $\BA$, and $(\mathcal{O},\Gamma \circ \Delta^*(\beta(\BA),V(\BA),V(\BA)))$ is the corresponding caterpillar duality, which has the following properties.
\begin{theorem}
$\BC(\BA) = \Gamma \circ \Delta^*(\beta(\BA),V(\BA),V(\BA))$ has caterpillar duality, and for any $\sigma$-structure $\BB$ with caterpillar duality, there exists a homomorphism
of $\BA$ to $\BB$ if and only if there exists a homomorphism of $\BC(\BA)$ to $\BB$.
In particular, $\BA$ itself has caterpillar duality if and only if there exists a homomorphism of
$\BC(\BA)$ to $\BA$.
\end{theorem}
This is essentially the characterization obtained in~\cite{cdk}. Note that $\Delta^*$ and $\Gamma$ are both exponential constructions, so that $\BC$ is a doubly exponential construction.

With a slight modification, the same type of characterization also holds for caterpillar dualities with additional properties. The most distinctive case is that of path dualities, where the obstructions are described by words not containing any of the symbols $R^{(i,i)}$ such that $R \in \sigma$ has arity at least $2$. For a $\sigma$-structure $\BA$, let  $\mathcal{L}_{\BA}$ be the language describing the caterpillar obstructions to $\BA$,
and $\mathcal{L}_{P} \subseteq \sigma_2^*$ be the set of words not containing any of the symbols $R^{(i,i)}$ such that $R \in \sigma$ has arity at least $2$. Then $\mathcal{L}_{P}$ and $\mathcal{L}_{P} \cap \mathcal{L}_{\BA}$ are regular languages, hence with the construction $\Gamma$ we can build a structure $\BC_P(\BA)$ such that $\BA$ has path duality if and only if there exists a homomorphism of $\BC_P(\BA)$ to $\BA$. A similar statement holds for any intersection $\mathcal{L} \cap \mathcal{L}_{\BA}$, where
$\mathcal{L} \subseteq \sigma_2^*$ is a regular language.


\begin{thebibliography}{99}

\bibitem{bk} L. Barto, M. Kozik,  Constraint satisfaction problems of bounded width, {sl Proc. 50th IEEE Symp. Foundations of Computer Science, FOCS'09} (2009), 595--603.

\bibitem{bkl} A. Bulatov, A. Krokhin, B. Larose, Dualities for Constraint Satisfaction Problems,  {\sl Complexity of Constraints} {\bf LNCS 5250}  (2008), 93--124.

\bibitem{cdk} C. Carvalho, V. Dalmau, A. Krokhin, Caterpillar Duality for Constraint Satisfaction Problems, {\sl Proc. 23rd  IEEE Symp. on Logic in Computer Science LICS'08}  (2008),  307--316.

\bibitem{ett} P. L. Erd\H{o}s, C. Tardif, G. Tardos, On infinite-finite duality pairs of directed graphs, manuscript (2012).
    
\bibitem{eptt} P.L. Erd\H os, D. P\'alv\"olgyi, C. Tardif, G. Tardos. On infinite-finite tree-duality  pairs of relational structures.  manuscript (2012).

\bibitem{ft} J. Foniok, C. Tardif, Adjoint functors and tree duality, {\sl Discrete Mathematics and Theoretical Computer Science} {\bf 11} (2) (2009), 97--110.

\bibitem{fv} T. Feder, M. Vardi, The computational structure of monotone monadic SNP and constraint satisfaction: A study through Datalog and group theory, {\sl SIAM J. of Computing} {\bf 28} (1998), 57--104.

\bibitem{llt} B. Larose, C. Loten, C. Tardif, A Characterisation of first order definable constraint satisfaction problems, {\sl Log. Methods Comput. Sci.} {\bf 3}  (4) (2007), paper 4:6, (22 pp.)

\bibitem{nt} J. Ne\v{s}et\v{r}il, C. Tardif, Duality theorems for finite structures (Characterising gaps and good characterisations), {\sl J. Combin. Theory} (B) {\bf 80} (2000), 80--97.

\bibitem{pultr} A. Pultr, The right adjoints into the categories of relational systems. In {\sl Reports of the Midwest Category Seminar, IV} {\bf LNM 137} (1970), 100--113.

\bibitem{sipser} M. Sipser, {\it Introduction to the Theory of Computation}, PWS Publishing Company, Boston, 1997.


\end{thebibliography}
\end{document}